\documentclass[english]{article}

\usepackage{style}

\begin{document}
\title{Simple homotopy type of the Hamiltonian Floer complex}
\author{Sebastian Pöder Balkest\r{a}hl}
\date{\today}

\maketitle

\begin{abstract}
	For an aspherical symplectic manifold, closed or with convex contact boundary, and with vanishing first Chern class, a Floer chain complex is defined for Hamiltonians linear at infinity with coefficients in the group ring of the fundamental group. For two non-degenerate Hamiltonians of the same slope continuation maps are shown to be simple homotopy equivalences. As a corollary the number of contractible Hamiltonian orbits of period $1$ can be bounded from below.
\end{abstract}

\section{Introduction}
Let $(\M, \omega)$ be a symplectic manifold, closed or the completion of a compact manifold $\LdomainM$ with convex boundary (see e.g. \cite{EliGro91}). 
Assume that $\omega|_{\pi_2 (\M)}$ and $c_1(M)|_{\pi_2 (\M)}$ vanish. For a non-degenerate Hamiltonian $H \fr \R /\Z \times \M \to \R$, linear at infinity if $\M$ is open, and an almost complex structure satisfying a regularity assumption, let $CF^*(H)$ denote the Hamiltonian Floer complex with coefficients twisted by the fundamental group of $\M$, reviewed in Section  \ref{sec:floer_setup}.

\begin{thm}\label{thm:mainthm}
	Let $(G,J_G)$ and $(H,J_H)$ be regular pairs where the Hamiltonians $G, H$ are linear at infinity with the same slope. Then a continuation map 
	\[
		CF^*(G) \to CF^*(H),
	\]
	induced by any regular homotopy from $G$ to $H$ and from $J_G$ to $J_H$, is a simple homotopy equivalence.
\end{thm}

Simple homotopy equivalence is an improvement compared to Floer's result of homotopy equivalence \cite{Floer88unreg}. Simple homotopy and the related notion of torsion has been studied for Floer complexes for instance in \cite{Sullivan02} \cite{Lee05} \cite{AK16} \cite{Suarez17} \cite{DRG18}. 
The complex $CF^*(H)$ also appears in \cite{OnoPaj16}.
In particular for the case of closed $\M$, Theorem \ref{thm:mainthm} and Corollary \ref{cor:stable_Morse_bound} appears in \cite{DRG17}, which is based on \cite{Sullivan02}. The latter performs bifurcation analysis using a geometric stabilization while here we use action-energy arguments.

Theorem \ref{thm:mainthm} is proved in Section \ref{sec:floer_setup} as Proposition \ref{prop:conclusion}.

The complex $CF^*(H)$ is generated by contractible $1$-periodic orbits of the Hamiltonian flow of $H$. Thus their number is constrained by the minimum rank among complexes with the simple homotopy type of $CF^*(H)$. If the slope of $H$ is small, this is the simple homotopy type of a Morse complex on $\LdomainM$, since for $H$ a small Morse function with small slope Floer's identification \cite{Floer88} of Floer and Morse complexes holds.

Damian  \cite{Damian02} shows that any complex which is simple homotopic to a Morse complex on $\LdomainM$ may be realized as the Morse complex of a stable Morse function on $\LdomainM$: a Morse function $\LdomainM \times \R^{2k} \to \R$ which is a compact perturbation of a Morse function on $\LdomainM$ plus a non-degenerate quadratic form on $\R^{2k}$. See Theorem \ref{thm:Damian_thm}. The stable Morse number of $\LdomainM$ is the minimum number of critical points of such functions. Thus
\begin{cor}\label{cor:stable_Morse_bound}
	Suppose that $H \fr \R/\Z \times \M \to \R$ is non-degenerate and linear at infinity with sufficiently small slope. 
Then the number of contractible 1-periodic orbits of the Hamiltonian vector field of $H$ is at least equal to the stable Morse number of $\LdomainM$.
\end{cor}

The Arnold conjecture gives the Morse number of $\LdomainM$ as a lower bound (the minimum number of critical points among Morse functions respecting the boundary). Clearly this would be better than Corollary \ref{cor:stable_Morse_bound}, and indeed there are examples due to Damian \cite{Damian02} where the Morse and stable Morse numbers differ. These examples depend only on the fundamental group; by work of Gompf \cite{Gompf94} any finitely presented group may be realized as the fundamental group of a symplectic manifold.

Examples of closed symplectic manifolds with $\pi_2 = 0$, hence satisfying the assumptions of Theorem \ref{thm:mainthm}, are given in \cite{DRG17}, and there exists also  examples of Gompf \cite{Gompf98} with $\pi_2 \neq 0$ but $\omega|_{\pi_2} = c_1 |_{\pi_2} = 0$ .

Here we note that the stable Morse number can be strictly larger than any bound coming from the cohomology of $\M$. Firstly, base change using the augmentation $\epsilon \fr \Lambda = \Z [\pi_1(\M)] \to \Z$ which map all group elements to $1$ sends the complex $CF^*(H)$ to the Floer complex $CF^*(H;\Z)$ with integer coefficients, which for small slopes computes the cohomology of $\M$ shifted down by $n=\frac12 \dim \M$. Secondly, any complex in the simple homotopy type of $CF^*(H)$ has a module in degree $1-n$ of rank at least equal to the minimal number $\delta$ of generators  of the $\Lambda$-module $\ker \epsilon$ \cite{Damian02}. Hence if for instance the fundamental group of $\M$ is nontrivial, finite, and perfect ($\pi_1 = [\pi_1, \pi_1]$) then $\delta \geq 2$ \cite{Damian02}, and hence there must be at least two generators in degree $1-n$ not seen in cohomology. For example we may find a Weinstein domain with $1$- and $2$-handles \cite{Weinstein91} attached according to a presentation of the desired group.

I thank my advisor Thomas Kragh for discussions and comments on previous versions of this document.

\section{Simple-homotopic chain complexes and stable Morse numbers}

General references for simple-homotopy theory are \cite{Whitehead50} or \cite{Cohen73}, but here the presentation is closer to \cite{AK16} and \cite{Damian02}.

Let $\Lambda$ be the ring $\Z [\pi_1(\LdomainM)]$, and let $C^*$ be a 
free and finitely generated complex over $\Lambda$, with a basis $\underline c$ which is a union of bases $\underline c_i$ of $C^i$. The simple homotopy type of $(C^*, \underline c)$ is the class of such based complexes under the equivalence relation generated by basis-preserving isomorphism and the following moves. 
\begin{enumerate}
	\item A trivial summand $0 \to \Lambda^r \xrightarrow{\mathrm{id}} \Lambda^r \to 0$, with preferred basis at both positions the standard basis of $\Lambda^r$, can be added or removed.
	\item The basis $\underline c_i$ can be replaced by a basis $\underline c'_i$ so that the change of basis matrix is either the identity matrix plus one off-diagonal element from $\Lambda$, or the identity matrix where one diagonal element is replaced by $\pm g \in \pm \pi_1(\LdomainM)$.
\end{enumerate}
Now let $C^*, D^*$ be free and finitely generated complexes with preferred bases $\underline c$ and $\underline d$. Write $C^*[1]$ for the shift 
  $(C[1])^k = C^{k+1}$. A chain homotopy equivalence $f \fr C^* \to D^*$ is \emph{simple} if the cone of $f$, 
\[
C_f = \left ( C^{*}[1]\oplus D^*, d = \begin{pmatrix} -d_C & 0 \\ f & d_D \end{pmatrix} \right ),
\] 
with the basis $\underline c \cup \underline d$, has the same simple homotopy type as the zero complex. If $f$ is simple, $C^*$ and $D^*$ have the same simple homotopy type \cite{Whitehead50}.

Next suppose that $C^*$ carry an increasing filtration $F^p C^* \subset F^{p+1} C^*$ with associated quotient groups $G^pC^* = F^p C^* / F^{p-1} C^*$, and similarly for $D^*$. The following is essentially Corollary 2.4 in \cite{AK16}. 
\begin{lma}\label{lma:alg_lma}
	Suppose that 
each filtration level $F^p C^*, F^p D^*$ is free on a subset of the 
bases $\underline c, \underline d$,
 and suppose that $ f\fr C^* \to D^*$ is an equivalence respecting the filtration. Moreover suppose that each of the induced maps $f^p \fr G^p C \to G^p D$ are simple homotopy equivalences. Then $f$ is a simple homotopy equivalence.
\end{lma}
\begin{proof}
	The complex $B^* := C_f$ has a filtration $F^p C_f = (F^p C^*[1]) \oplus F^p D^*$ with associated quotients $G^p B^* = G^pC_f = C_{f^p}$.  
Let $k$ be the highest filtration level, so $F^{k-1}B^* \subsetneq F^kB^* = B^*$. The complex $G^kB^*$ may by assumption be reduced to zero by a sequence of moves, which we will lift to reduce $F^k B^* $ to $F^{k-1} B^*$.

	Only when removing a trivial summand do we encounter any difficulity. Suppose basis elements $r \in F^k B^i, s \in F^k B^{i+1}$ span, in $G^k B^*$, a trivial summand with $d \bar r = \bar s$. However in $F^k B^*$ we may have that $d r = s + t$, where $t \in F^{k-1}B^*$. We replace the basis element $s\in F^k B^*$ with $s' = s+t \in F^k B^*$. Thus $d r = s'$ and this new basis still satisfies the assumptions. Moreover, let $r^\perp \subset F^k B^i$ be the span of the other basis elements in this degree and similarly $s^\perp \subset F^k B^{i+1}$. Since the differential of $G^k B^*$ takes $r^\perp /F^{k-1} B^i$ to $s^\perp / F^{k-1} B^{i+1}$, we have $d r^\perp \subset s^\perp$, too. Hence $r, s'$ span a trivial summand of $F^kB^*$, which we then may remove.

	Thus $F^k B^*$ may be reduced to $F^{k-1} B^*$. Proceeding inductively we can reduce $F^k B^* = C_f$ to the zero complex.	
\end{proof}

Let $(f, V)$ be a Morse-Smale pair on $\LdomainM$; if there is boundary, we assume that $V$ points out. By choosing a preimage in the universal cover of $\LdomainM$ for each critical point and an orientation of the descending manifold, we may construct  Morse complex $C^*(f,V)$ over the ring $\Lambda$. Namely if $\gamma$ is a rigid flow-line with $\gamma(-\infty) = x, \gamma(\infty) =y$, we lift $\gamma(\infty)$ to the chosen lift of $y$; the induced lift of $\gamma(-\infty)$ is then the chosen lift of $x$ acted on by some $g \in \pi_1(\LdomainM)$ and $\gamma$ contributes $\pm gx$ to $\partial y$. The simple homotopy type of the complex $C^*(f,V)$ is independent of the pair $(f,V)$.

When $\LdomainM$ is closed, the arguments of M. Damian \cite{Damian02} apply to show that the minimum rank of complexes in the simple homotopy type of $C^*(f,V)$ is equal to the stable Morse number $\mu_{st}$ of $\LdomainM$. This is the minimum number of critical points among Morse functions $h \from \LdomainM \times \R^{k} \times \R^k \to \R$ which, outside a compact set, agree with the quadratic form $Q(m, x, y) = |x|^2 - |y|^2$. If $W$ is a gradient-like vector field for $h$ which outside a compact set equals the gradient of $Q$ we say that $(h,W)$ is standard at infinity.

For $\LdomainM$ with boundary, the argument of Damian can be adopted slightly as in Proposition 2.9 in \cite{DRG18} to conclude similarly. Here the stable Morse functions $h$ are required to be of the form $h= Q + f$ outside a compact set in the interior, where $f$ is a function on $\LdomainM$ having $\partial \LdomainM$ contained in a regular level set. Similarly a gradient-like vector field $W$ of $h$ is required to point out of along $\partial \LdomainM \times \R^{2k}$ and outside a compact set be of the form $\nabla Q + V$, $V$ a vector field on $\LdomainM$. 
\begin{thm}[\cite{Damian02}, \cite{DRG17}] \label{thm:Damian_thm}
	Let $D^*$ be a free finitely generated based $\Lambda$-complex in the simple homotopy type of $C^* (f,V)$. There is a $k \in \N$, a stable Morse function $h \from \LdomainM \times \R^{k} \times \R^k \to \R$, and a gradient-like vector field $W$ such that $(h,W)$ is Morse-Smale and standard at infinity, and such that
	\[
		D^* = C^* (h,W)[k]. 
	\]
\end{thm}

\begin{proof}
	We merely emphasize the important points of the arguments of \cite{Damian02}, \cite{DRG17}. Choose $k$ such that $D^*[-k]$ is supported in the degrees $3, \dotsc, \dim \LdomainM + 2k -3$. One starts with a pair $(f,V)$ on $\LdomainM$ and stabilize to $(f+Q, V + \nabla Q)$. Then, Morse moves are performed on the latter to mimic the algebraic manipulations to get from $ C^* (f+Q, V + \nabla Q) = C^*(f,V)[-k]$ to $D^*[-k]$. These Morse moves are raising/lowering the value at a critical point, performing handle slides, add a pair of critical points which form a trivial subcomplex, cancel two trajectories which together form a nullhomotopic loop, and cancel a pair of critical points which form a trivial subcomplex and with only one trajectory between them. 

To perform these we need a little room. Let $F \subset \LdomainM$ be the union of all critical points of $f$ and any trajectories between them, and let $U \subset \LdomainM \times \R^{2k}$ be an open set containing $F \times 0$ and contained in a compact set $K \subset \operatorname{int} (\LdomainM ) \times \R^{2k}$. All moves can be performed in $U$, so that the resulting pair $(h,W)$ equals $(f+Q, V + \nabla Q)$ outside $U$. We may choose $K$ such that any trajectory which exits $K$ in forwards or backwards time never enters it again, instead going to infinity or to the boundary $\partial \LdomainM \times \R^{2k}$. From these considerations the arguments of Damian \cite{Damian02} produces $(h,W)$ as desired.
\end{proof}

\section{Hamiltonian Floer setup and simple homotopy}\label{sec:floer_setup}

 For a linear at infinity function $H\fr \R/\Z \times \M \to \R$ we define $X = X_{H_t}$ by $\omega(-, X)  = dH_t$. The function $H$ is \emph{admissible} if all $1$-periodic orbits of $X$ are non-degenerate, in particular the slope at infinity is not equal to the length of any Reeb orbit. We choose a cylindrical almost complex structure $J= J_t$ compatible with $\omega$ and such that $(H,J)$ is a regular pair. 

Let $\mcP$ be the finite set of contractible $1$-periodic orbits of $X$. For any filling disc $w_{x}$ of $x \in \mcP$, we set the action of $x$ as 
\[
	\mcA_{H} (x) = \int_{D^2} w_x^* \omega - \int_0^1 H(t, x(t)) dt. 
\]
which, since $\omega |_{\pi_2} =0$, is independent of the filling. Grading $x$ by the Conley-Zehnder index $|x|$ is also independent of the filling as $c_1|_{\pi_2} = 0$. The Floer equation
\begin{equation}\label{eq:floereq}
 	\partial_s u + J (u) \partial_t u + \nabla H (u) = 0
\end{equation}
represents the positive gradient flow of $\mcA_H$. Fix for each $x\in \mcP$ a lift of $x$ to the cover $\widetilde \M$, and let $CF^*(H) $ be freely generated over $\Lambda$ by $\mcP$ with grading $|\cdot |$. For $x\in \mcP$, the differential $\partial x$ counts rigid solutions of \eqref{eq:floereq} with input $x$ at $s=+\infty$: as in the Morse case lift the curve to $\widetilde \M$ with initial condition the chosen lift of $x$; then as $s\to -\infty$ it limits to a lift of some $y \in \mcP$; we compare with the chosen lift for $y$ and record the curve's contribution as $\pm g y $ for some $g\in \pi_1(\M)$. The sign is determined using the coherent orientations of \cite{FH93}.

A regular path $(H^s, J^s)$, constantly equal to $(H^\pm, J^\pm)$ for $\pm s \gg 0$,  induces a continuation map $\Phi \fr CF^* (H^+) \to CF^*(H^-)$ as long as the slope of $H^s$ is increasing as $s$ decreases. Whenever the slope is constant, or if $\M$ is compact, $\Phi$ is a chain equivalence. 

In the latter case $\Phi$ is a simple chain equivalence, which we will argue in this section. We may assume that the path $(H^s, J^s)$ is such that there is a dense set of $s' \in \R$ where the pairs $(H^{s'}, J^{s'})$ are regular, and hence the complexes $CF^*(H^{s'})$ are defined; moreover we may assume that for all $s$ the sets $\mcP (H^s)$ are finite. If $s' < s''$ are sufficiently close we will show that a continuation map $CF^*(H^{s''}) \to CF^*(H^{s'})$, induced by the path $H^s$, is a simple equivalence.  The map $\Phi$ is homotopic to a composition of such maps and hence also simple \cite{Cohen73}. Let us note some properties of continuation maps like $\Phi$. If the $y$-coefficient of $\Phi (x)$ is nonzero, there is a curve $u$ which solves \eqref{eq:floereq} with $J^s, H^s$ in place of $J,H$, and where $u(s, \cdot)$ has  limits $x$ respectively $y$ as $s\to \infty$ or $\to -\infty$. The energy 
\[
	E(u) = \frac12 \int_\R \int_0^1 | \partial_s u |^2 + |\partial_t u - X (u) |^2 dt ds \geq 0
\]
of such a curve satisfies 
\begin{equation} \label{eq:cont_energy}
	E(u) = \mcA_{H^+} (x) -\mcA_{H^-} (y) + \int_\R \int_0^1 (\partial_s H^s_t ) (u(s,t)) dt ds.
\end{equation}
Hence if $\partial_s H^s$ is small, the action of $y$ is at most slightly greater than the action of $x$. Now fix $s_0 \in \R$. We may assume that $H^s$ is non-degenerate for all $s$ near $s_0$ except possibly for $H^{s_0}$ \cite{Oh02}. Let $ b_1 < \dotsb < b_k$ be the action values of $\mcP (H^{s_0})$, and pick $d>0$  less than each $(b_{i+1}-b_i)/4$. Consider $s_\pm$ close to $s_0$ with $s_- < s_0 < s_+$ such that the complexes $CF^*(H^{s_{\pm}})$ are defined. If they are sufficiently close, the action values of $\mcP ({H^{s_\pm}})$ lie in a $d$-neighbourhood of $\{ b_1 , \dotsc b_k \}$, and 
\[
	\int_{s_-}^{s_+} \int_{S^1} \max_{x \in \M} | \partial_s H^s_t (x)| dt ds < d.
\]
For such $s_\pm$ we filter the complexes $CF^*(H^{s_\pm})$ by setting $F^p CF^* (H^{s_\pm})$ to be generated by those orbits in $\mcP (H^{s_\pm})$ of action $\leq (b_p + b_{p+1}) /2$; $F^0$ is the zero complex, and $F^k$ is the full complex. Now, let $\eta (s) $ denote a function increasing from $s_-$ to $s_+$.  From the path $(H^s, J^s)$ we get a path $(H^\eta,  J^\eta)$ which induces a continuation map $\phi \fr CF^*(H^{s_+}) \to CF^*(H^{s_-})$. Due to the choice of $d$, $\phi$ preserves the filtration and induces maps of associated groups
\[
	\phi^p \fr G^p (H^{s_+}) = F^p CF^* (H^{s_+}) / F^{p-1} CF^* (H^{s_+}) \to G^p (H^{s_-})
\]
which are also equivalences. From Lemma \ref{lma:alg_lma}, we need only show that each $\phi^p$ is simple. 

Now each $x\in \mcP (H^{s_\pm})$ tend to a unique $z \in \mcP(H^{s_0})$ as $s_\pm \to s_0$; let $B_z^\pm$ be the collection of those tending to $z$, and which therefore have action close to $\mcA_{H^{s_0}} (z)$. Thus we have a splitting of modules
\begin{equation}\label{eq:splitting}
	G^p(H^{s_\pm}) \cong \bigoplus_{z \in \mcP(H^{s_0}); \mcA_{H^{s_0}} = b_p} \Lambda B_z^\pm.
\end{equation}
 A key point is to show that this is a splitting of subcomplexes which the continuation maps preserve, compare \cite{AK16}. 

It is convenient to modify the chosen lifts for elements of $ B_z^\pm$; namely, after choosing a lift of $z$ we prefer to choose the lift of $x \in B^\pm_z$ which is close to the lift of $z$. In the chain complex this is a change of basis $x \to g x$ which does not affect whether an equivalence is simple. This is assumed for the next proposition.

\begin{prop}\label{prop:small_step_simple}
	There is $\delta > 0 $ such that if $| s_\pm - s_0 | < \delta$, the splitting \eqref{eq:splitting} is a splitting of subcomplexes, and $\phi^p (B^+_z) \subset \Lambda B^-_z$ for each $z$. Futhermore  with respect to the preferred basis,  the matrices of the restrictions $\partial^\pm \fr \Lambda B^\pm_z \to \Lambda B^\pm_z$,  $ \phi^p \fr \Lambda B^+_z \to \Lambda B^-_z$ have integer coefficients.
\end{prop}
\begin{rmk}\label{rmk:nonregular}
	If the pairs $(H^\pm, J^\pm)$ or the path $(H^\eta, J^\eta)$ are non-regular, the conclusion still holds for regular pairs or paths sufficiently close. 
\end{rmk}
The proof follows Lemma \ref{lma:compactness}. This gives us
\begin{prop}\label{prop:conclusion}
	The continuation map $\Phi \fr CF^* (H^+) \to CF^*(H^-)$ is a simple homotopy equivalence. 
\end{prop}
\begin{proof}
	The $s$-dependence of $(H^s, J^s)$ is a compact set $A \subset \R$. We can choose finitely many $s_i$ and an $\delta_{\mathrm{min}} >0$ such that the intervals $(s_i - \delta_\mathrm{min}, s_i + \delta_\mathrm{min})$ cover $A$ and Proposition \ref{prop:small_step_simple} applies whenever $\delta< \delta_\mathrm{min}$.

Choose $r_i \in (s_{i+1} - \delta_\mathrm{min}, s_i + \delta_\mathrm{min})$. By Remark \ref{rmk:nonregular} we may assume that $(H^{r_i}, J^{r_i})$ respectively the induced paths $(H^{\eta_i}, J^{\eta_i})$ are regular, and there are thus equivalences
\[
	\phi_i \fr CF^* (H^{r_{i+1}}) \to CF^* (H^{r_i}).
\]
These are simple: from Proposition \ref{prop:small_step_simple} we may regard the maps $\phi_i^p$ of associated groups as maps of $\Z$-complexes with extended scalars; maps of $\Z$-complexes are always simple and the same holds for extensions of simple maps \cite{Cohen73}. Thus by Lemma \ref{lma:alg_lma} so are the $\phi_i$. Hence there is a map
\[
	\Phi' \fr CF^*(H^+) \to CF^*(H^-)
\]
which is formed by the composition of the simple $\phi_i$. Being a composition of simple equivalences, $\Phi'$ is simple, too \cite{Cohen73}. Finally note that $\Phi$ is homotopic to $\Phi'$ and thus also simple \cite{Cohen73}.
\end{proof}
 
It remains to prove Proposition \ref{prop:small_step_simple}. The situation is as follows. We have the path $(H^s,J^s)$ and the chosen value $s_0$. Let $\sigma(s,b) $, $b \in [0,1]$, be a smooth homotopy from the constant function $s \mapsto s_0$ to a function increasing from $s_0-1$ to $s_0+1$ such that 
\begin{align*}
	\sigma(s_0,b) & = s_0 \mbox{ for all } b \\
	\sigma(s,b)& = s_0 - b \mbox{ for } s<s_0 - 1 \mbox{ and all } b \\
	\sigma(s,b) & = s_0 + b \mbox{ for } s> s_0+1 \mbox{ and all } b \\
	\partial_s \sigma & \geq 0.
\end{align*}

Consider the family $H^{\sigma (s,b)}$, suppressing the almost complex structure $J^{\sigma(s,b)}$. 

Let $\delta>0$ be so small that 
  if $\gamma$ is an orbit of $X^s$ with $|s-s_0|<\delta$ it makes sense to say that $\gamma$ lies close to some unique orbit of $X^{s_0}$. Let $\mcP(s) = \mcP (H^s)$.
 We will use the following notation for moduli spaces: $\msppartial{gx}{y}{H}$ contains curves contributing $\pm gx$ to $\partial y$, whenever the dimension is correct, and similarly the spaces $\mspspartial{gx}{y}{H^s}$ defines the image of $y$ under the continuation map. 
\begin{lma}\label{lma:compactness}
	Let $H^s, J^s$ be a given admissible path, fix $s_0 \in \R$, and let $\delta$ be as above. There is a $\beta \in (0, \delta)$, such that for any $b\in (0,\beta)$, any $x, y \in \mcP (s_0) $ with $\mcA_{s_0}(x ) = \mcA_{s_0}(y)$, and any $x'$ close to $x$, $y'$ close to $y$, the following spaces are empty.
\begin{enumerate}
	\item For $x\neq y$, 
	 $x' \in \mcP(s_0 - b), y' \in \mcP(s_0 + b)$, $\mspspartial{ gx'}{ y'}{H^{\sigma(s,b)}} = \emptyset$.
	\item For $x\neq y$, $x', y' \in \mcP(s_0 \pm b) $, $\msppartial{gx'}{ y'}{H^{s_0 \pm b}} = \emptyset$.
 	\item For $x=y$, $x' \in \mcP(s_0 - b), y' \in \mcP(s_0 + b)$, $\mspspartial{g  x'}{   y'} 
{H^{\sigma(s,b)}} = \emptyset$, unless $g=1$.
	\item For $x=y$, $x',y' \in \mcP(s_0 \pm b) $, $\msppartial{g  x'}{ y'}{H^{s_0  \pm b}} = \emptyset$, unless $g=1$.
\end{enumerate}
In the latter two cases the lifts of $ x',  y'$ are chosen such that the lifts of  $ x'(0) , y'(0),$ and $x(0)$ lie in the same sheet above $x(0)$.
\end{lma}

\begin{proof}
Consider part 1. Let $x, y \in \mcP (s_0)$ be different and fix a small ball $B$ at $x(0)$ such that the closure of $B$ does not contain $z(0)$ for any $z \in \mcP(s_0) \setminus{x}$. Consider a sequence $b_n\to 0$ and corresponding maps $u_n \in \mspspartial{x_n}{y_n}{H^{\sigma(s,b_n)}}$, where the ends are close to $x$ respectively $y$ as in the Lemma. Thus $E(u_n) \to 0$. Let $s_n$ be the smallest number for which $u_n(s_n, 0)$ hits $\partial B$. The shifted sequence $w_n = u_n(\cdot +s_n, \cdot)$ solves \eqref{eq:floereq} for data $H^{\sigma(s + s_n, b_n)}$, which still converges to $H^{s_0}$ as $n\to \infty$.

We apply Gromov compactness to the sequence $w_n$, so that after passing to a subsequence we find a limit $v \fr \R \times S^1 \to \M$ which solves \eqref{eq:floereq} with data $H^{s_0}$; since $E(v) \leq \lim E(w_n) = 0$ we must have $v(s,t) = z(t) $ for some $z \in \mcP(H^{s_0})$. But  $w_n$ satisfies $w_n(0,0)\in \partial B$, hence $z(0) = v (0,0) \in \partial B$, which contradicts the choice of $B$.

Hence there is no such sequence $u_n$. We conclude that there is $\beta(x,y) > 0$ such that for $b \in (0,\beta)$ there are no solutions of \eqref{eq:floereq} between any $x' \in \mcP(s_0 - b)$ near $x$ and any $y' \in \mcP(s_0 + b)$ near $y$. The set $\mcP(s_0) $ is finite, so we can take $\beta$ as the minimum among $\beta(x,y)$. This proves part 1. Part 2 is similar.

Assume that the situation is the same as above but that $y=x\in \mcP(s_0)$. Consider a $u_n$, and suppose that the path $s \mapsto u_n(s,0)$ is homotopic, relative ends, to a path contained in the neighbourhood $B$ of $x(0)$. In this case there is nothing to prove, so we may assume that it is not so for any $u_n$. Then the hitting times $s_n$ are well defined and the argument proceeds as above. This proves parts 3 and 4.
\end{proof}

\begin{proof}[Proof of proposition \ref{prop:small_step_simple}.]
Take $\delta < \beta$. 
By perturbing slightly, we may assume that the conclusions of Lemma \ref{lma:compactness} hold and that the pairs $(H^\pm, J^\pm)$ and the path  $(H^\eta, J^\eta) =(H^{\sigma(s,b)}, J^{\sigma( s, b)})$ are regular.

Then part 2 of Lemma \ref{lma:compactness} gives that the splitting \eqref{eq:splitting} is a splitting of subcomplexes, and part 1 that $\phi$ preserves it. Parts 3 and 4 give that the matrices of the maps $\partial^+, \partial^-,$ and $\phi$ have integer coefficients. 
\end{proof}

\bibliographystyle{abbrv}
\bibliography{../bib/researchdefault}

\end{document}